\documentclass[12pt]{amsart}
\usepackage{amsmath,amssymb,amsthm}
\usepackage{enumitem}
\usepackage{graphicx}
\usepackage{amsfonts}
\usepackage{bigints}
\usepackage{amsfonts}
\usepackage{mathptmx}  

\usepackage{epstopdf}
\newcommand{\overbar}[1]{\mkern 1.5mu\overline{\mkern-1.5mu#1\mkern-1.5mu}\mkern 1.5mu}

\theoremstyle{plain}
\newtheorem{theorem}{Theorem}[section]

\theoremstyle{definition}

\theoremstyle{Conjecture}
\newtheorem{Conjecture}[theorem]{Conjecture}
\makeatletter
\@addtoreset{equation}{section}
\makeatother

\makeatletter
\newcommand{\Spvek}[2][r]{%
	\gdef\@VORNE{1}
	\left(\hskip-\arraycolsep%
	\begin{array}{#1}\vekSp@lten{#2}\end{array}%
	\hskip-\arraycolsep\right)}

\def\vekSp@lten#1{\xvekSp@lten#1;vekL@stLine;}
\def\vekL@stLine{vekL@stLine}
\def\xvekSp@lten#1;{\def\temp{#1}%
	\ifx\temp\vekL@stLine
	\else
	\ifnum\@VORNE=1\gdef\@VORNE{0}
	\else\@arraycr\fi%
	#1%
	\expandafter\xvekSp@lten
	\fi}
\makeatother
\begin{document}
	\title[Jacobi matrices for fractal measures]{Asymptotic properties of Jacobi matrices for a family of fractal measures}
	\author{G\"{o}kalp Alpan}
\address{Department of Mathematics, Bilkent University, 06800 Ankara, Turkey}
\email{gokalp@fen.bilkent.edu.tr}
\thanks{The authors are partially supported by a grant from T\"{u}bitak: 115F199.}
	\author{Alexander Goncharov}
\address{Department of Mathematics, Bilkent University, 06800 Ankara, Turkey}

\email{goncha@fen.bilkent.edu.tr}

	\author{Ahmet N\.{i}hat  \c{S}\.{i}m\c{s}ek}
	\address{Department of Mathematics, Bilkent University, 06800 Ankara, Turkey}
	
	\email{	nihat.simsek@bilkent.edu.tr}


\begin{abstract}
We study the properties and asymptotics of the Jacobi matrices associated with equilibrium measures of the weakly equilibrium Cantor sets. These family of Cantor sets were defined and different aspects of orthogonal polynomials on them were studied recently. Our main aim is numerically examine some conjectures concerning orthogonal polynomials which do not directly follow from previous results. We also compare our results with more general conjectures made for recurrence coefficients associated with fractal measures supported on $\mathbb{R}$.


\end{abstract}
\maketitle
\section{Introduction}

For a unit Borel measure $\mu$ with an infinite compact support on $\mathbb{R}$, using the Gram-Schmidt process for the set $\{1,x,x^2,\ldots\}$ in $L^2(\mu)$, one can find a sequence of polynomials $(q_n(\cdot;\mu))_{n=0}^\infty$ satisfying $$\int q_m(x;\mu)q_n(x;\mu)\,d\mu(x)=\delta_{mn}$$ where $q_n(\cdot;\mu)$ is of degree $n$. Here, $q_n(\cdot;\mu))$ is called the $n$-th orthonormal polynomial for $\mu$. We denote its positive leading coefficient  by $\kappa_n$ and $n$-th monic orthogonal polynomial $q_n(\cdot;\mu)/\kappa_n$ by $Q_n(\cdot;\mu)$. If we assume that $Q_{-1}(\cdot;\mu):=0$ and $Q_0(\cdot;\mu):=1$ then there are two bounded sequences $(a_n)_{n=1}^\infty$, $(b_n)_{n=1}^\infty$ such that the polynomials $(Q_n(\cdot;\mu))_{n=0}^{\infty}$
satisfy a three-term recurrence relation
$$Q_{n+1}(x;\mu) = (x- b_{n+1})Q_{n}(x;\mu) - a_n^2 \, Q_{n-1}(x;\mu),\,\,\,\,\,\,\,\,n \in \mathbb{N}_0, $$ where $a_n>0$, $b_n\in\mathbb{R}$ and $\mathbb{N}_{0}=\mathbb{N}\cup \{0\}$. 

Conversely, if two bounded sequences $(a_n)_{n=1}^\infty$ and $(b_n)_{n=1}^\infty$ are given with $a_n>0$ and $b_n\in\mathbb{R}$ for each $n\in\mathbb{N}$ then we can define the corresponding Jacobi matrix $H$, which is a self-adjoint bounded operator acting on $l^2(\mathbb{N})$, as the following,
\begin{equation}
H=\left( \begin{array}{ccccc}
b_1 & a_1 &0 & 0 &\ldots \\
a_1 & b_2 & a_2 & 0& \ldots \\
0& a_2 & b_2 & a_3 & \ldots \\
\vdots & \vdots & \vdots & \vdots&\ddots \\
\end{array} \right).
\end{equation}

The (scalar valued) spectral measure $\mu$ of $H$ for the cyclic vector $(1,0,\ldots)^T$ is the measure that has $(a_n)_{n=1}^\infty$ and $(b_n)_{n=1}^\infty$ as recurrence coefficients. Due to this one to one correspondence between measures and Jacobi matrices, we denote the Jacobi matrix associated with $\mu$ by $H_\mu$. For a discussion of the spectral theory of orthogonal polynomials on $\mathbb{R}$ we refer the reader to \cite{Sim3,ase}.

Let $c=(c_n)_{n=-\infty}^\infty$ be a two sided sequence taking values on $\mathbb{C}$ and $c^j=(c_{n+j})_{n=-\infty}^\infty$ for $j\in\mathbb{Z}$. Then $c$ is called almost periodic if $\{c^j\}_{j\in\mathbb{Z}}$ is precompact in $l^{\infty}(\mathbb{Z})$. A one-sided sequence $d=(d_n)_{n=1}^\infty$ is called almost periodic if it is the restriction of a two sided almost periodic sequence to $\mathbb{N}$. Each one sided almost periodic sequence has only one extension to $\mathbb{Z}$ which is almost periodic, see Section 5.13 in \cite{Sim3}. Hence one-sided and two sided almost periodic sequences are essentially the same objects. A Jacobi matrix $H_\mu$ is called almost periodic if the sequences of recurrence coefficients $(a_n)_{n=1}^\infty$ and $(b_n)_{n=1}^\infty$ for $\mu$ are almost periodic.  We consider in the following sections only one-sided sequences due to the nature of our problems but, in general, for the almost periodicity, it is much more natural to consider sequences on $\mathbb{Z}$ instead of $\mathbb{N}$. 

A sequence $s= (s_n)_{n=1}^\infty$ is called asymptotically almost periodic if there is an almost periodic sequence $d=(d_n)_{n=1}^\infty$ such that $d_n-s_n\rightarrow 0$ as $n\rightarrow \infty$. In this case $d$ is unique and it is called the almost periodic limit.   See \cite{petersen, Sim3, teschl} for more details on almost periodic functions.

Several sufficient conditions on $H_\mu$ to be almost periodic or asymptotically almost periodic are given in \cite{peher,sodin} for the case when  $\mathrm{ess\, supp}(\mu)$ (that is the support of $\mu$ excluding its isolated points) is a Parreau-Widom set (Section 3) or in particular homogeneous set in the sense of Carleson (see \cite{peher} for the definition). We remark that some symmetric Cantor sets and generalized Julia sets (see \cite{peher,alpgon2}) are Parreau-Widom. By \cite{Barnsley4, yuditskii}, for equilibrium measures of some polynomial Julia sets corresponding Jacobi matrices are almost periodic. It was conjectured in \cite{mant2,kruger} that Jacobi matrices for self-similar measures including the Cantor measure are asymptotically almost periodic. We should also mention that some almost periodic Jacobi matrices with applications to physics (see e.g. \cite{avil}), has essential spectrum equal to a Cantor set.

There are many open problems regarding orthogonal polynomials on Cantor sets, such as how to define the Szeg\H{o} class of measures and isospectral torus (see e.g. \cite{chriss, Chris} for the previous results and \cite{Heilman, kruger,mant1,mant3,mant4} for possible extensions of the theory and important conjectures) especially when the support has zero Lebesgue measure. The family of sets that we consider here contains both positive and zero Lebesgue measure sets, Parreau-Widom and non Parreau-Widom sets. Widom-Hilbert factors (see Section 2 for the definition) for equilibrium measures of the weakly equilibrium Cantor sets may be bounded or unbounded depending on the particular choice of parameters. Some properties of these measures related to orthogonal polynomials were already studied in detail but till now we do not have complete characterizations of most of the properties mentioned above in terms of the parameters. Our results and conjectures are meant to suggest some formulations of theorems for further work on these sets as well as other Cantor sets.

The plan of the paper is as follows. In Section 2, we review the previous results on $K(\gamma)$ and provide evidence for the numerical stability of the algorithm obtained in Section 4 in \cite{alpgon} for calculating the recurrence coefficients. In Section 3, we discuss the behavior of recurrence coefficients in different aspects and propose some conjectures about the character of periodicity of the Jacobi matrices. In Section 4, the properties of Widom factors are investigated. We also prove that the sequence of Widom-Hilbert factors for the equilibrium measure of autonomous quadratic Julia sets is unbounded above as soon as the Julia set is totally disconnected. In the last section, we study local behavior of the spacing properties of the zeros of orthogonal polynomials for the equilibrium measures of weakly equilibrium Cantor sets and make a few comments on possible consequences of our numerical experiments.

For a general overview on potential theory we refer the reader to \cite{Ransford,saff}. For a non-polar compact set $K\subset \mathbb{C}$, the equilibrium measure is denoted by $\mu_K$ while $\mathrm{Cap}(K)$ stands for the logarithmic capacity of $K$. The Green's function for the connected component of $\overbar{\mathbb{C}}\setminus K$ containing infinity is denoted by $G_K(z)$. Convergence of measures is understood as weak-star convergence. For the sup norm on $K$ and for the Hilbert norm on $L^2(\mu)$ we use $\|\cdot\|_{L^{\infty}(K)}$ and $\|\cdot\|_{L^{2}(\mu)}$ respectively.

\section{Preliminaries and numerical stability of the algorithm}

Let us repeat the construction of $K(\gamma)$ which was introduced in \cite{gonc}. Let $\gamma=(\gamma_s)_{s=1}^\infty$ be a sequence such that $0<\gamma_s<1/4$ holds for each $s\in\mathbb{N}$ provided that $\sum_{s=1}^\infty 2^{-s}\log{(1/{\gamma_s)}}<\infty$. Set $r_0=1$ and $r_s=\gamma_s r_{s-1}^2$. We define $(f_n)_{n=1}^\infty$ by $f_1(z):=2z(z-1)/\gamma_1+1$ and $f_n(z):=z^2/(2\gamma_n)+1-1/(2\gamma_n)$ for $n>1$. Here $E_0:=[0,1]$ and $E_n:=F_n^{-1}([-1,1])$ where $F_n$ is used to denote $f_n\circ\dots\circ f_1$. Then, $E_n$ is a union of $2^n$ disjoint non-degenerate closed intervals in $[0,1]$ and $E_n\subset E_{n-1}$ for all $n\in\mathbb{N}$. Moreover, $K(\gamma):=\cap_{n=0}^\infty E_n$ is a non-polar Cantor set in $[0,1]$ where $\{0,1\}\subset K(\gamma)$. It is not hard to see that for each different $\gamma$ we end up with a different $K(\gamma)$. 

It is shown in Section 3 of \cite{alpgon} that for all $s\in\mathbb{N}_0$ we have

 \begin{equation} \label{norm1}
 ||Q_{2^s}\left(\cdot;\mu_{K(\gamma)}\right)||_{L^{2}\left(\mu_{K(\gamma)}\right)}=\sqrt{(1-2\,\gamma_{s+1})\,r_s^2/4}.
 \end{equation}
 
 The diagonal elements, the $b_n$'s of $H_{\mu_{K(\gamma)}}$, are equal to $0,5$ by Section 4 in \cite{alpgon}. For the outdiagonal elements by Theorem 4.3 in \cite{alpgon} we have the following relations:
 \begin{align}
\label{a1} a_1=\|Q_1\left(\cdot;\mu_{K(\gamma)}\right)\|_{L^{2}\left(\mu_{K(\gamma)}\right)},\\ 
\label{a2}a_2=\|Q_2\left(\cdot;\mu_{K(\gamma)}\right)\|_{L^{2}\left(\mu_{K(\gamma)}\right)}/\|Q_1\left(\cdot;\mu_{K(\gamma)}\right)\|_{L^{2}\left(\mu_{K(\gamma)}\right)}.
 \end{align}
 If $n+1=2^s>2$  then
 \begin{equation}\label{rec1}
a_{n+1}=\frac{||Q_{2^s}\left(\cdot;\mu_{K(\gamma)}\right)||_{L^{2}\left(\mu_{K(\gamma)}\right)}}{||Q_{2^{s-1}}\left(\cdot;\mu_{K(\gamma)}\right)||_{L^{2}\left(\mu_{K(\gamma)}\right)} \cdot a_{2^{s-1}+1}\cdot a_{2^{s-1}+2}\cdots a_{2^{s}-1}}.
 \end{equation}
If $n+1=2^s(2k+1)$ for some $s\in\mathbb{N}$ and $k\in\mathbb{N},$ then 
\begin{equation}\label{rec2}
a_{n+1}=\sqrt{\frac{\|Q_{2^s}\left(\cdot;\mu_{K(\gamma)}\right)\|_{L^{2}\left(\mu_{K(\gamma)}\right)}^2-a_{2^{s+1}k}^2\cdots a_{2^{s+1}k-2^s+1}^2}
{ a_{2^s(2k+1)-1}^2\cdots a_{2^{s+1}k+1}^2}},
\end{equation}
If $n+1=(2k+1)$ for $k\in\mathbb{N}$ then 
\begin{equation}\label{rec3}
a_{n+1}=\sqrt{\|Q_{1}\left(\cdot;\mu_{K(\gamma)}\right)\|_{L^{2}\left(\mu_{K(\gamma)}\right)}^2-a_{2k}^2}.
\end{equation}

The relations \eqref{norm1}, \eqref{a1}, \eqref{a2}, \eqref{rec1}, \eqref{rec2},  \eqref{rec3} completely determine $(a_n)_{n=1}^\infty$ and naturally define an algorithm. This is the main algorithm that we use and we call it Algorithm 1. There are a couple of results for the asymptotics of $(a_n)_{n=1}^\infty$, see Lemma 4.6 and Theorem 4.7 in \cite{alpgon}.

We want to examine numerical stability of Algorithm 1 since roundoff errors can be huge due to the recursive nature of it. Before this, let us list some remarkable properties of $K(\gamma)$ which will be considered later on. In the next theorem one can found proofs of part $(a)$ in \cite{alpgon3}, $(b)$ and $(c)$  in \cite{alpgon}, $(d)$ and $(e)$ in \cite{alpgon2},  $(f)$ in \cite{alpgonhat}, $(g)$ in \cite{gonc} and $(h)$ and $(i)$ in \cite{g1}. We call $W_n^2(\mu):= \frac{\|Q_n(\cdot;\mu)\|_{L^2(\mu)}}{(\mathrm{Cap}(\mathrm{supp}(\mu)))^n}$ as the $n$-th Widom-Hilbert factor for $\mu$.

\begin{theorem}\label{bigtheorem}
For a given $\gamma=(\gamma_s)_{s=1}^\infty$ let $\varepsilon_s:=1-4\gamma_s$. Then the following propositions hold:
	\begin{enumerate}[label={(\alph*})]
		\item If $\sum_{s=1}^\infty \gamma_s<\infty$ and $\gamma_s\leq 1/32$ for all $s\in\mathbb{N}$ then $K(\gamma)$ is of Hausdorff dimension zero.
		\item If $\gamma_s\leq 1/6$ for each $s\in\mathbb{N}$ then $K(\gamma)$ has zero Lebesgue measure, $\mu_{K(\gamma)}$ is purely singular continuous and $\liminf a_n=0$ for $\mu_{K(\gamma)}.$
		\item Let $\tilde{f}:=(\tilde{f}_s)_{s=1}^\infty$ be a sequence of functions such that $\tilde{f}_s=f_s$ for $1\leq s\leq k$ for some $k\in\mathbb{N}$ and $\tilde{f}_s(z)=2z^2-1$ for $s>k$. Then $\cap_{n=1}^\infty \tilde{F}_n^{-1}([-1,1])=E_k$ where $\tilde{F}_n:=\tilde{f}_n\circ\dots\circ \tilde{f}_1$.
		\item $G_{K(\gamma)}$ is H\"{o}lder continuous with exponent $1/2$ if and only if $\sum_{s=1}^\infty \varepsilon_s<\infty$.
		\item $K(\gamma)$ is a Parreau-Widom set if and only if $\sum_{s=1}^\infty \sqrt{\varepsilon_s}<\infty$.
		\item If $\sum_{s=1}^\infty \varepsilon_s<\infty$ then there is $C>0$ such that for all $n\in\mathbb{N}$ we have $$W_n^2(\mu_{K(\gamma)})= \frac{\|Q_n\left(\cdot;\mu_{K(\gamma)}\right)\|_{L^{2}\left(\mu_{K(\gamma)}\right)}}{(\mathrm{Cap}(K(\gamma)))^n}=\frac{a_1\dots a_n}{(\mathrm{Cap}(K(\gamma)))^n}\leq Cn.$$
		\item $\mathrm{Cap}(K(\gamma))=\exp{(\sum_{k=1}^{\infty}2^{-k}\log{\gamma_k})}.$ 
		\item Let $v_{1,1}(t)=1/2-(1/2)\sqrt{1-2\gamma_1+2\gamma_1 t}$ and $v_{2,1}(t)= 1-v_{1,1}(t)$. For each $n>1$, let $v_{1,n}(t)=\sqrt{1-2\gamma_n+2\gamma_n t}$ and $v_{2,n}(t)=-v_{1,n}(t)$. Then the zero set of $Q_{2^s}\left(\cdot;\mu_{K(\gamma)}\right)$ is $\{v_{i_1,1}\circ\dots \circ v_{i_s,s}(0)\}_{i_s\in\{1,2\}}$ for all $s\in\mathbb{N}$.
		\item $\mathrm{supp}(\mu_{K(\gamma)})=\mathrm{ess\, supp}(\mu_{K(\gamma)})=K(\gamma)$. If $K(\gamma)=[0,1]\setminus\cup_{k=1}^\infty (c_i,d_i)$ where $c_i\neq d_j$ for all $i,j\in\mathbb{N}$ then $\mu_{K(\gamma)}([0,e_i])\subset \{m2^{-n}\}_{m,n\in\mathbb{N}}$ where $e_i\in(c_i,d_i)$. Moreover for each $m\in\mathbb{N}$ and $n\in\mathbb{N}$ with $m2^{-n}<1$ there is an $i\in\mathbb{N}$ such that  $\mu_{K(\gamma)}([0,e_i])=m2^{-n}$.
	\end{enumerate}
\end{theorem}

We consider $4$ different models depending on $\gamma$ in the whole article. They are:
	\begin{enumerate}
		\item $\gamma_s= 1/4- (1/(50+s)^4).$
		\item $\gamma_s= 1/4- (1/(50+s)^2).$
		\item $\gamma_s= 1/4- (1/(50+s)^{(5/4)}.$
		\item $\gamma_s= 1/4- (1/50).$
\end{enumerate}
Model $1$ represents an example where $K(\gamma)$ is Parreau-Widom and Model $2$ gives a non Parreau-Widom set with fast growth of $\gamma$. Model $3$ produces a non Parreau-Widom $K(\gamma)$ with relatively slow growth of $\gamma$ but still $G_{K(\gamma)}$ is optimally smooth. Model $4$ yields a set which is neither Parreau-Widom nor the Green's function for the complement of it is optimally smooth. We used Matlab in all of the experiments.

If $f$ is a nonlinear polynomial having real coefficients with real and simple zeros $x_1<x_2<\ldots<x_n$ and distinct extremas $y_1<\ldots<y_{n-1}$ where $|f(y_i)|>1$ for $i=1,2, \ldots, n-1$, we say that $f$ is an \emph{admissible} polynomial. Clearly, for any choice of $\gamma$, $f_n$ is admissible for each $n\in\mathbb{N}$ and this implies by Lemma 4.3 in \cite{alpgon2} that $F_n$ is also admissible. By the remark after Theorem 4 and Theorem 11 in \cite{van assche} it follows that the Christoffel numbers (see p. 565 in \cite{van assche} for the definition) for the $2^n$-th orthogonal polynomial of $\mu_{E_n}$ are equal to $1/2^n$. Let $\mu^n_{K(\gamma)}$ be the measure which assigns $1/2^n$ mass to each zero of $Q_{2^n}(\cdot;\mu_{K(\gamma)})$. From Remark 4.8 in \cite{alpgon} the recurrence coefficients $(a_k)_{k=1}^{2^n-1}$, $(b_k)_{k=1}^{2^n}$ for $\mu_{E_n}$ are exactly those of $\mu_{K(\gamma)}$. This implies that (see e.g. Theorem 1.3.5 in \cite{Sim3}) the Christoffel numbers corresponding to $2^n$-th orthogonal polynomial for $\mu_{K(\gamma)}$ are also equal to $1/2^n$. 
\begin{figure}
	\centering
	\includegraphics[scale=.6]{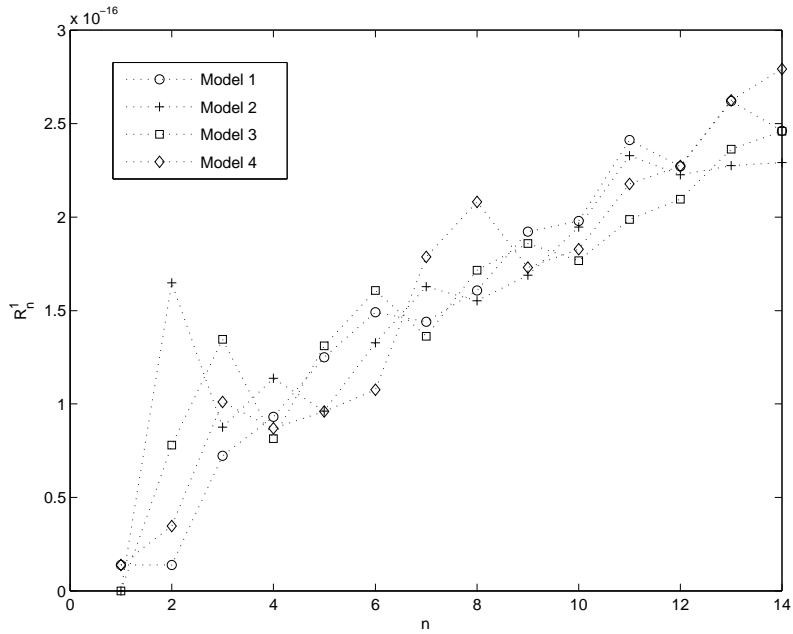}
	\caption{Errors associated with eigenvalues.}
	\label{fig:fig1}
\end{figure} 
\begin{figure}[!htb]
	\centering
	\includegraphics[scale=.5]{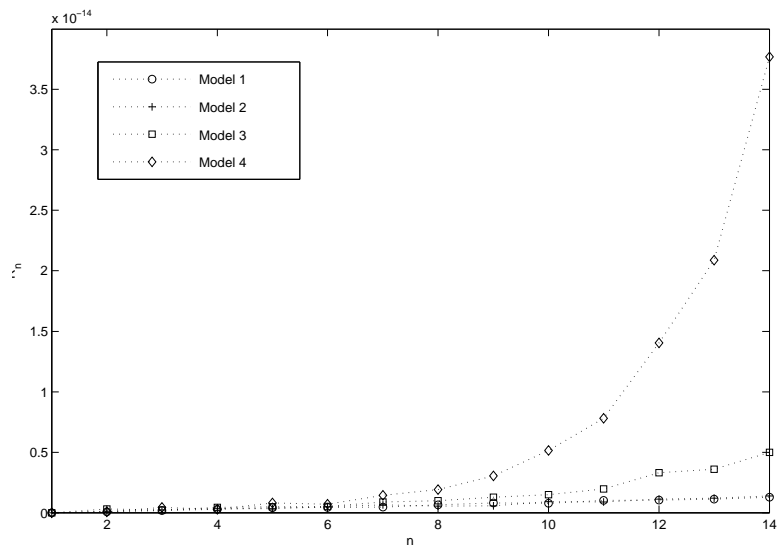}
	\caption{Errors associated with eigenvectors.}
	\label{fig:fig2}
\end{figure}

Let 
\begin{equation}
H_{\mu^n_{K(\gamma)}}=\left( \begin{array}{ccccc}
b_1 & a_1 \\
a_1 & b_2 & a_2 \\
& a_2 & \ddots & \ddots \\
& & \ddots & \ddots & a_{2^n-1} \\
& & & a_{2^n-1} & b_{2^n}

\end{array} \right),
\end{equation}

where the coefficients $(a_k)_{k=1}^{2^n-1}$, $(b_k)_{k=1}^{2^n}$ are the Jacobi parameters for $\mu_{K(\gamma)}$. Then the set of eigenvalues of $H_{\mu^n_{K(\gamma)}}$ is exactly the zero set of $Q_{2^n}\left(\cdot;\mu_{K(\gamma)}\right)$. Moreover, by \cite{golub}, the square of first component of normalized eigenvectors gives one of the Christoffel numbers, which in our case is equal to $1/2^n$. For each $n\in\{1,\dots,14\}$, using gauss.m, we computed the eigenvalues and first component of normalized eigenvectors of $H_{\mu^n_{K(\gamma)}}$ where the coefficients are obtained from Algorithm 1. We compared these values with the zeros obtained by part $(h)$ of Theorem \ref{bigtheorem} and $1/2^n$ respectively. For each $n$, let $\{t_k^n\}_{k=1}^{2^n}$ be the set of eigenvalues for $H_{\mu^n_{K(\gamma)}}$ and $\{q_k^n\}_{k=1}^{2^n}$ be the set of zeros where we enumerate these sets so that the smaller the index they have, the value will be smaller. Let $\{w_k^n\}_{k=1}^{2^n}$ be the set of squared first component of normalized eigenvectors. We plotted (see Figure ~\ref{fig:fig1} and Figure ~\ref{fig:fig2}) $R_n^1:=(1/2^n)(\sum_{k=1}^{2^n}|t_k^n-q_k^n|)$ and $R_n^2:=(1/2^n)(\sum_{k=1}^{2^n}|(1/2^n)-w_k^n|)$. This numerical experiment shows the reliability of Algorithm 1. One can compare these values with Fig. 2 in \cite{mant4}.

\section{Recurrence Coefficients}
\begin{figure}[!htb]
	\centering
	\includegraphics[scale=.6]{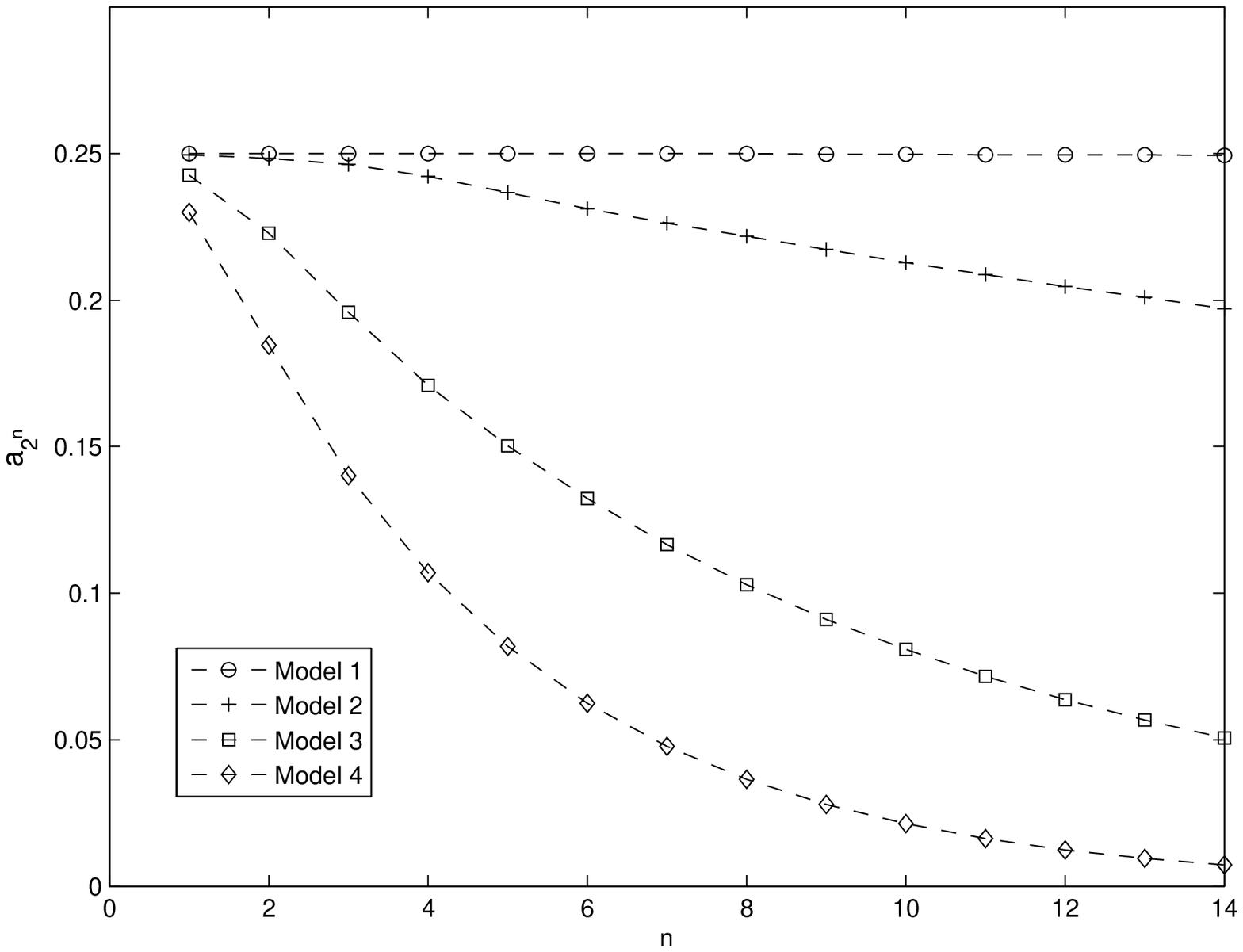}
	\caption{The values of outdiagonal elements of Jacobi matrices at the indices of the form $2^s$.}
	\label{fig:fig3}
\end{figure}
\begin{figure}[!htb]
	\centering
	\includegraphics[scale=.6]{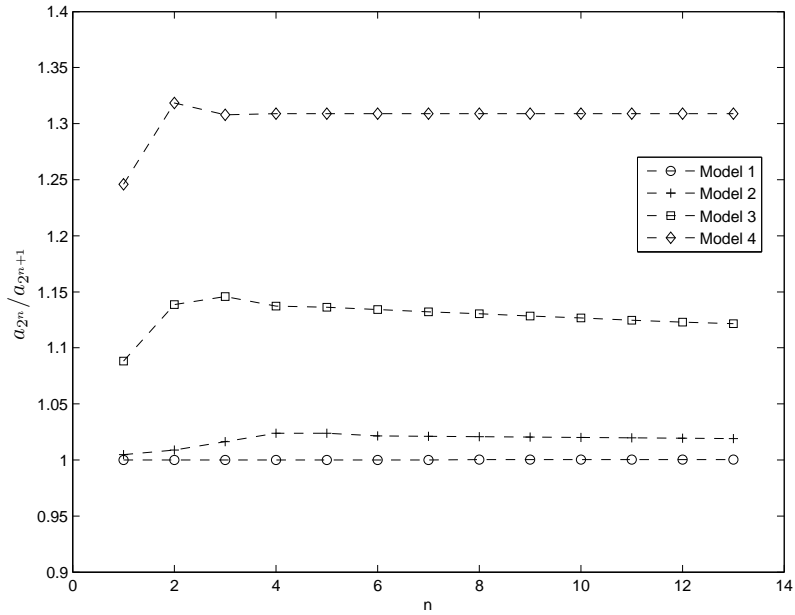}
	\caption{The ratios of outdiagonal elements of Jacobi matrices at the indices of the form $2^s$.}
	\label{fig:fig4}
\end{figure}
\begin{figure}[!htb]
	\centering
	\includegraphics[scale=.6]{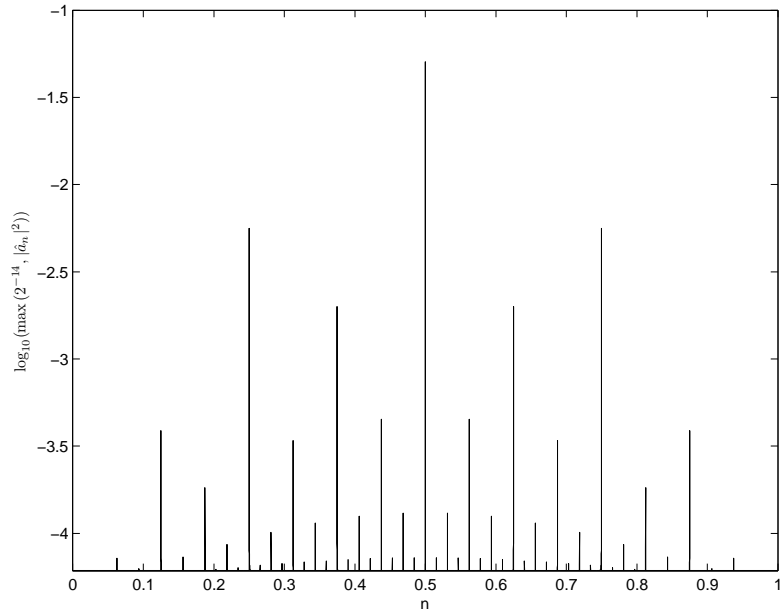}
	\caption{Normalized power spectrum of the $a_n$'s for Model 1.}
	\label{fig:fig5}
\end{figure}

It was shown (for the stretched version of this set but similar arguments are valid for this case also) in \cite{alpgon2} that $K(\gamma)$ is a generalized polynomial Julia set (see e.g. \cite{Bruck1, Bruck,Buger} for a discussion on generalized Julia sets) if $\inf \gamma_k>0$, that is $K(\gamma):=\partial \{z\in\overbar{\mathbb{C}}: F_n(z)\rightarrow \infty \mbox{ locally uniformly}\}$. Let $J(f)$ be the (autonomous) Julia set for $f(z)=z^2-c$ for some $c>2$. Since $(f_n)_{n=1}^\infty$ is a sequence of quadratic polynomials, it is natural to ask that to what extent  $H_{\mu_{J(f)}}$ and $H_{\mu_{K(\gamma)}}$ have similar behavior. Compare for example Theorem 4.7 in \cite{alpgon} with Section 3 in \cite{besger}.

The recurrence coefficients for $\mu_{J(f)}$ can be ordered according to their indices, see (IV.136)-(IV.138) in \cite{bes2}. We obtain similar results for $\mu_{K(\gamma)}$ in our numerical experiments in each 4 models. That is the numerical experiments suggest that $\displaystyle\min_{i\in\{1,\dots, 2^n\}}{a_i}=a_{2^{n}}$ for $n\leq 14$ and it immediately follows from \eqref{a1} and \eqref{rec3} that $\displaystyle \max_{n\in\mathbb{N}} a_n =a_1$. Thus, we make the following conjecture:

\begin{Conjecture}\label{conj1}
	For $\mu_{K(\gamma)}$ we have $\displaystyle\min_{i\in\{1,\dots, 2^n\}}{a_i}=a_{2^{n}}$ and in particular  $\displaystyle \,\,\liminf_{s\rightarrow\infty}a_{2^s}=\liminf_{n\rightarrow\infty}a_{n}.$
\end{Conjecture}

A non-polar compact set $K\subset\mathbb{R}$ which is regular with respect to the Dirichlet problem is called a Parreau-Widom set if $\sum_{k=1}^\infty G_K(e_k)<\infty$ where $\cup_k e_k$ is the set of critical points, which is at most countable, of $G_K$. Parreau-Widom sets have positive Lebesgue measure. It is also known that (see e.g. Remark 4.8 in \cite{alpgon}) $\liminf a_n>0$ for ${\mu_K}$ provided that $K$ is Parreau-Widom. For more on Parreau-Widom sets, we refer the reader to \cite{christiansen,yuditskii}.

By part $(e)$ of Theorem \ref{bigtheorem}, $\liminf a_n>0$ for $\mu_{K(\gamma)}$ provided that $\sum_{s=1}^\infty \sqrt{\varepsilon_s}<\infty $. It also follows from Remark 4.8 in \cite{alpgon} and \cite{dombrowski} that if the $a_n$'s associated with $\mu_{K(\gamma)}$ satisfy $\liminf a_n=0$ then $K(\gamma)$ has zero Lebesgue measure. Hence asymptotic behavior of the $a_n$'s is also important for understanding the Hausdorff dimension of $K(\gamma)$. We computed $v_n: =a_{2^{n}}/a_{2^{n+1}}$ (see Figure ~\ref{fig:fig3} and Figure ~\ref{fig:fig4}) for $n=1,\dots, 13$ in order to find for which $\gamma$'s $\liminf a_n=0$. We assume here Conjecture \ref{conj1} is correct.

In Model 1, $v_n$ is very close to $1$ which is expected since for this case $\liminf a_n>0$. In other models, it seems that $(v_n)_{n=1}^{13}$ seems to behave like a constant. Thus, this experiment may be read as unless $\sum_{s=1}^\infty \sqrt{\varepsilon_s}<\infty$ is satisfied $\liminf a_n=0$. So, we conjecture:

\begin{Conjecture}\label{conj2}
For a given $\gamma=(\gamma_k)_{k=1}^\infty$, let $\varepsilon_k:= 1-4\gamma_k$ for each $k\in\mathbb{N}$. Then $K(\gamma)$ is of positive Lebesgue measure if and only if $\sum_{s=1}^\infty \sqrt{\varepsilon_s}<\infty$ if and only if $\liminf a_n>0$.
\end{Conjecture}

A more interesting problem is whether $H_{\mu_{K(\gamma)}}$ is almost periodic or at least asymptotically almost periodic. Since $(b_n)_{n=1}^\infty$ is a periodic sequence, we only need to deal with $(a_n)_{n=1}^\infty$.

For a measure $\mu$ with an infinite compact support $\mathrm{supp}(\mu)$, let $\delta_n$ be the normalized counting measure on the zeros of $Q_n(\cdot;\mu)$. If there is a $\nu$ such that $\delta_n\rightarrow \nu$ then $\nu$ is called the density of states (DOS) measure for $H_\mu$. Besides, $\int_{-\infty}^x d\nu$ is called the integrated density of states (IDS). For $H_{\mu_{K(\gamma)}}$ the density of states measure is automatically (see Theorem 1.7 and Theorem 1.12 in \cite{Sim3} and also \cite{widom}) $\mu_{K(\gamma)}$. Therefore, if $x$ is chosen from one of the gaps (by a gap of a compact set on $K\subset\mathbb{R}$ we mean a bounded component of $R\setminus K$) of $\mathrm{supp}\left(\mu_{K(\gamma)}\right)$, that is $x\in (c_i,d_i)$ (see part (i) of Theorem \ref{bigtheorem}) then the value of the IDS is equal to $m2^{-n}$ which does not exceed $1$ and also for each $m,n\in\mathbb{N}$ with $m2^{-n}<1$  there is a gap $(c_j,d_j)$ such that the IDS takes the value $m2^{-n}$.

For an almost periodic sequence $c=(c_n)_{n=1}^\infty$ the $\mathbb{Z}$-module of the real numbers modulo $1$ generated by $\omega$ satisfying $$ \{\omega:\, \lim_{n\rightarrow\infty} \frac{1}{N} \sum_{n=1}^{N} \exp{(2\pi in\omega)} c_n\neq 0\}$$ is called the frequency module for $c$ and it is denoted by $\mathcal{M}(c)$. The frequency module is always countable and $c$ can be written as a uniform limit of Fourier series where the frequencies are chosen among $\mathcal{M}(c)$. For an almost periodic Jacobi matrix $H$ with coefficients $a=(a_n)_{n=1}^\infty$ and $b=(b_n)_{n=1}^\infty$, the frequency module $\mathcal{M}(H)$ is the module generated by $\mathcal{M}(a)$ and $\mathcal{M}(b)$. It was shown in Theorem III.1 in \cite{delyon} that for an almost periodic $H$, the values of IDS in gaps belong to $\mathcal{M}(H)$. Moreover, (see e.g. Theorem 2.4 in \cite{gerhar}), an asymptotically almost periodic Jacobi matrix has the same density of states measure with the almost periodic limit of it.

In order to examine almost periodicity of the $a_n$'s for $\mu_{K(\gamma)}$ we computed the discrete Fourier transform $(\widehat{a}_n)_{n=1}^{2^{14}}$ for the first $2^{14}$ coefficients for each model where frequencies run from $0$ to $1$. We normalized $|\widehat{a}|^2$ dividing it by $\sum_{n=1}^{2^{14}} |\widehat{a}_n|^2$. We plotted (see Figure ~\ref{fig:fig5}) this normalized power spectrum while we did not plot the peak at $0$, by detrending the transform.

There are only a small number of peaks in each case compared to $2^{14}$ frequencies which points out almost periodicity of coefficients. We consider only Model 1 here although we have similar pictures for the other models. The highest 10 peaks are at $0.5, 0.25,  0.75, 0.375, 0.625,$ $0.4375, 0.5625, 0.125, 0.875, 
0.3125$. All these values are of the form $m2^{-n}$ where $n\leq4$. This is an important indicator of almost periodicity as these frequencies are exactly the values of IDS for $H_{\mu_{K(\gamma)}}$ in the gaps which appear earlier in the construction of the Cantor set. The following conjecture follows naturally from the above discussion.

\begin{Conjecture}\label{conj3}
For any $\gamma$, $(a_n)_{n=1}^\infty$ for $H_{\mu_{K(\gamma)}}$ is asymptotically almost periodic where the almost periodic limit has frequency module equal to $\{m2^{-n}\}_{m,n\in\{N_0\}}$ modulo 1.
\end{Conjecture}

\section{Widom factors} 
\begin{figure}[!htb]
	\centering
	\includegraphics[scale=.6]{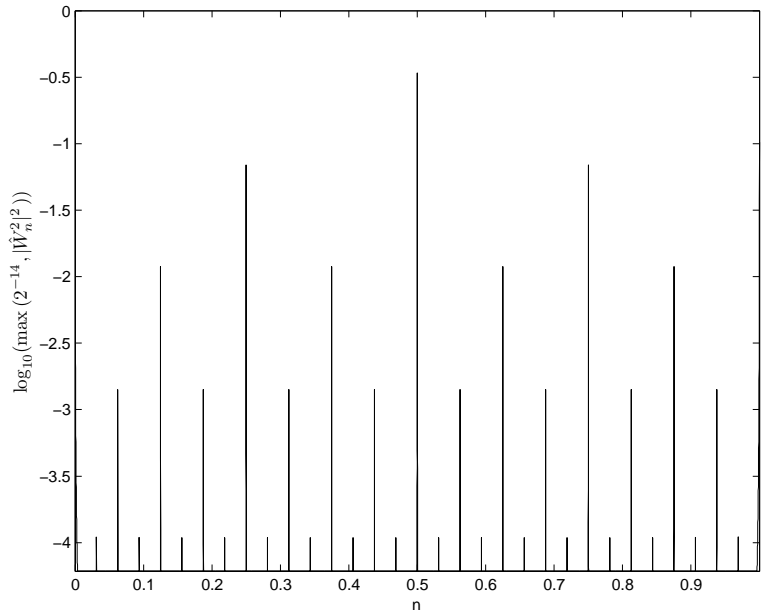}
	\caption{Normalized power spectrum of the $W_n^2\left(\mu_{K(\gamma)}\right)$'s for Model 1.}
	\label{fig:fig6}
\end{figure}
\begin{figure}[!htb]
	\centering
	\includegraphics[scale=.6]{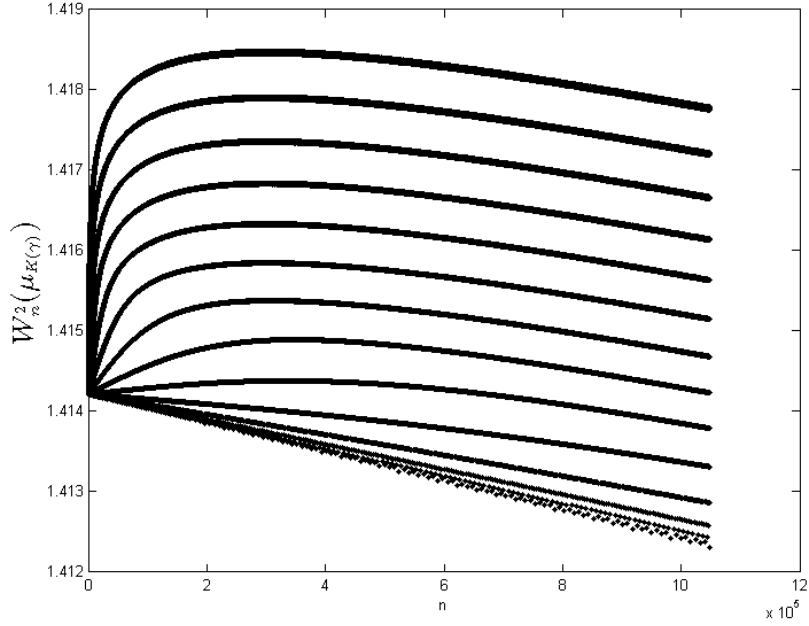}
	\caption{Widom-Hilbert factors for Model 1}
	\label{fig:fig7}
\end{figure}

Let $K\subset\mathbb{C}$ be a non-polar compact set. Then the unique monic polynomial $T_n$ of degree $n$ satisfying $$\|T_n\|_{{L^\infty}(K)}=\min\{\|P_n\|_{{L^\infty}(K)}: \mbox{$P_n$ complex monic polynomial of degree $n$}\}$$ is called the $n$-th \emph{Chebyshev polynomial} on $K$ where $\|\cdot \|_{{L^\infty}(K)}$ is the sup-norm on $K$. 

We define the $n$-th Widom factor for the sup-norm on $K$ by  $W_n(K)=||T_n||_{{L^\infty}(K)}/(\mathrm{Cap}(K))^n$. It is due to Schiefermayr \cite{sif} that $W_n(K)\geq 2$ if $K\subset \mathbb{R}$. It is also known that (see e.g. \cite{fekete,szego}) $\|T_n\|_{{L^\infty}(K)}^{1/n}\rightarrow \mathrm{Cap}(K)$ as $n\rightarrow\infty$. This implies a theoretical constraint on the growth rate of $W_n(K)$, that is $(1/n)\log{W_n(K)}\rightarrow 0$ as $n\rightarrow\infty$.
See for example \cite{tot11,tot2,tot1} for further discussion.

Theorem 4.4 in \cite{gonchat} says that for each sequence $(M_n)_{n=1}^\infty$ satisfying $\lim_{n\rightarrow\infty}(1/n)\log{M_n}=0$, there is a $\gamma$ such that $W_n(K(\gamma))>M_n$. On the other hand, for many compact subsets of $\mathbb{C}$ (see e.g. \cite{andr,csz,totikvar,widom2}) the sequence of Widom factors for the sup-norm is bounded. In particular, this is valid for Parreau-Widom sets on $\mathbb{R}$, see \cite{csz}. It would be interesting to find (if any) a non Parreau-Widom set $K$ on $\mathbb{R}$ such that it is regular with respect to the Dirichlet problem and $(W_n(K))_{n=1}^\infty$ is bounded. Note that if $K$ is a non-polar compact subset of $\mathbb{R}$ which is regular with respect to the Dirichlet problem then by Theorem 4.2.3 in \cite{Ransford} and Theorem 5.5.13 in \cite{Sim3} we have $\mathrm{supp}(\mu_K)= K$. In this case, we have $W_n^2(\mu_K)\leq W_n(K)$ since $\|Q_n(\cdot;\mu_K)\|_{L^2(\mu_K)}\leq \|T_n\|_{L^2(\mu_K)}\leq \|T_n\|_{{L^\infty}(K)}$. Therefore, it is possible to formulate the above problem in a weaker form: Is there a non Parreau-Widom set $K\subset\mathbb{R}$ which is regular with respect to the Dirichlet problem such that $(W_n^2(\mu_K))_{n=1}^\infty$ is bounded?

In \cite{alpgon4}, the authors following \cite{Barnsley3} studied  $\left(W_n^2\left(\mu_{J(f)}\right)\right)_{n=1}^\infty$ where $f(z)=z^3-\lambda z$ for $\lambda>3$ and showed that the sequence is unbounded. For this particular case the Julia set is a compact subset of $\mathbb{R}$ which has zero Lebesgue measure. It is always true for a polynomial autonomous Julia set $J(f)$ on $\mathbb{R}$ that $\mathrm{supp}\left(\mu_{J(f)}\right)=J(f)$ since $J(f)$ is regular with respct to the Dirichlet problem by \cite{Mane}. Now, let us show that $\left(W_n^2\left(\mu_{J(f)}\right)\right)_{n=1}^\infty$ is unbounded when $f(z)=z^2-c$ and $c>2$. These quadratic Julia sets are zero Lebesgue measure Cantor sets on $\mathbb{R}$ and therefore not Parreau-Widom. See \cite{Brolin} for a deeper discussion on this particular family.

\begin{theorem}\label{teo}
	Let $f(z)=z^2-c$ for $c\geq 2$. Then $\left(W_n^2\left(\mu_{J(f)}\right)\right)_{n=1}^\infty$ is bounded if and only if $c=2$.
\end{theorem}

\begin{proof}
	If $c=2$ then $J(f)=[-2,2]$. This implies that $\left(W_n^2\left(\mu_{J(f)}\right)\right)_{n=1}^\infty$ is bounded since $J(f)$ is Parreau-Widom.
	
	Let $c\neq 2$. Then $\lim_{n\rightarrow\infty}a_{2^n}=0$ (see e.g. Section IV.5.2 in \cite{bes2}) where the $a_n$'s are the recurrence coefficients for $\mu_{J(f)}$ and $\mathrm{Cap}(J(f))=1$ by \cite{Brolin}. Since $Q_{2^{n+1}}\left(\cdot;\mu_{J(f)}\right)=Q_{2^n}^2\left(\cdot;\mu_{J(f)}\right) -c$ by Theorem 3 in \cite{Barnsley1}, we have $W_{2^n}^2\left(\mu_{J(f)}\right)=\|Q_{2^n}\left(\cdot;\mu_{J(f)}\right)\|_{L^2\left({\mu_{J(f)}}\right)}=\sqrt{c}$ for all $n\geq 1$. Moreover,
	\begin{equation}
	 W_{2^n-1}^2\left(\mu_{J(f)}\right)= \frac{W_{2^n}^2\left(\mu_{J(f)}\right)}{a_{2^n}}=\frac{\sqrt{c}}{a_{2^n}}.
	\end{equation} 
	Hence $\lim_{n\rightarrow\infty}W_{2^n-1}^2\left(\mu_{J(f)}\right)=\infty$ as $\lim_{n\rightarrow\infty}a_{2^n}=0$. This completes the proof.
\end{proof}

In \cite{alpgon}, it was shown that $\left(W_n^2\left(\mu_{K(\gamma)}\right)\right)_{n=1}^\infty$ is unbounded if $\gamma_k\leq 1/6$ for all $k\in\mathbb{N}$. We want to examine the behavior of $\left(W_n^2\left(\mu_{K(\gamma)}\right)\right)_{n=1}^\infty$ provided that $K(\gamma)$ is not Parreau-Widom. By \cite{alpgon}, $\left(W_{2^n}\left(\mu_{K(\gamma)}\right)\right)\geq \sqrt{2}$ for all $n\in\mathbb{N}_0$ for any choice of $\gamma$. Hence,we also have 
\begin{equation}\label{eqeq}
W_{2^n-1}^2\left(\mu_{K(\gamma)}\right)=W_{2^n}^2\left(\mu_{K(\gamma)}\right)\frac{\mathrm{Cap}\left(\mu_{K(\gamma)}\right)}{a_{2^n}}\geq \frac{\sqrt{2}\mathrm{Cap}\left(\mu_{K(\gamma)}\right)}{a_{2^n}}
\end{equation} 
for all $n\in\mathbb{N}$. 

If we assume that Conjecture \ref{conj1} and Conjecture \ref{conj2} are correct then $\liminf_{n\rightarrow\infty}a_{2^n}=0$ as soon as $K(\gamma)$ is not Parreau-Widom. If $\liminf_{n\rightarrow\infty}a_{2^n}=0$ then $\limsup_{n\rightarrow\infty}W_{2^n-1}\left(\mu_{K(\gamma)}\right)=\infty$ by \eqref{eqeq}. Thus, the numerical experiments indicate the following:

\begin{Conjecture}\label{conj4}
$K(\gamma)$ is a Parreau-Widom set if and only if $\left(W_n^2\left(\mu_{K(\gamma)}\right)\right)_{n=1}^\infty$ is bounded if and only if $(W_n(K(\gamma)))_{n=1}^\infty$ is bounded.
\end{Conjecture}

Let $K$ be a union of finitely many compact non-degenerate intervals on $\mathbb{R}$ and $\omega$ be the Radon-Nikodym derivative of $\mu_K$ with respect to the Lebesgue meeasure on the line. Then $\mu_K$ satisfies the Szeg\H{o} condition: $\int_{K} \omega(x)\log{\omega(x)}\,dx>-\infty$. This implies by Corollary 6.7 in \cite{Chris} that $(W_n^2(\mu_K))_{n=1}^\infty$ is asymptotically almost periodic. If $K$ is a Parreau-Widom set, $\mu_K$ satisfies the Szeg\H{o} condition by \cite{pommerenke}. We plotted (see Figure ~\ref{fig:fig7}) the Widom-Hilbert factors for Model 1 for the first $2^{20}$ values and it seems that $\limsup W_n^2\left(\mu_{(K(\gamma))}\right)\neq \sup W_n^2\left(\mu_{(K(\gamma))}\right)$. For Model 1, we plotted (see Figure ~\ref{fig:fig6}) the power spectrum for $(W_n^2(\mu_K))_{n=1}^{2^{14}}$ where we normalized $|\widehat{W}^2|^2$ dividing it by $\sum_{n=1}^{2^{14}} |\widehat{W}^2_n(\mu_K)|^2$. Frequencies run from $0$ to $1$ here and we did not plot the big peak at $0$.

Clearly, there are only a few peaks as in (see Figure ~\ref{fig:fig5}) which is an important indicator of almost periodicity. The highest 10 peaks are at $0.5, 0.00006103515625, 0.25, 0.75, 0.125,$ $0.875, 0.375, 0.625, 0.062.5, 0.9375$. These values are quite different than those of peaks in Figure ~\ref{fig:fig5}. This may be an indicator of a different frequency module of the almost periodic limit. By Conjecture \ref{conj4}, $\left(W_n^2\left(\mu_{K(\gamma)}\right)\right)_{n=1}^\infty$ is unbounded and cannot be asymptotically almost periodic if $K(\gamma)$ is not Parreau-Widom. We make the following conjecture:

\begin{Conjecture}\label{conj5}
$\left(W_n^2\left(\mu_{K(\gamma)}\right)\right)_{n=1}^\infty$ is asymptotically almost periodic if and only if $K(\gamma)$ is Parreau-Widom. If $K(\gamma)$ is Parreau-Widom then the almost periodic limit's frequency module includes the module generated by $\{m2^{-n}\}_{m,n\in\{N_0\}}$ modulo 1.
\end{Conjecture}

\section{Spacing properties of orthogonal polynomials and further discussion}
\begin{figure}[!htb]
	\centering
	\includegraphics[scale=.6]{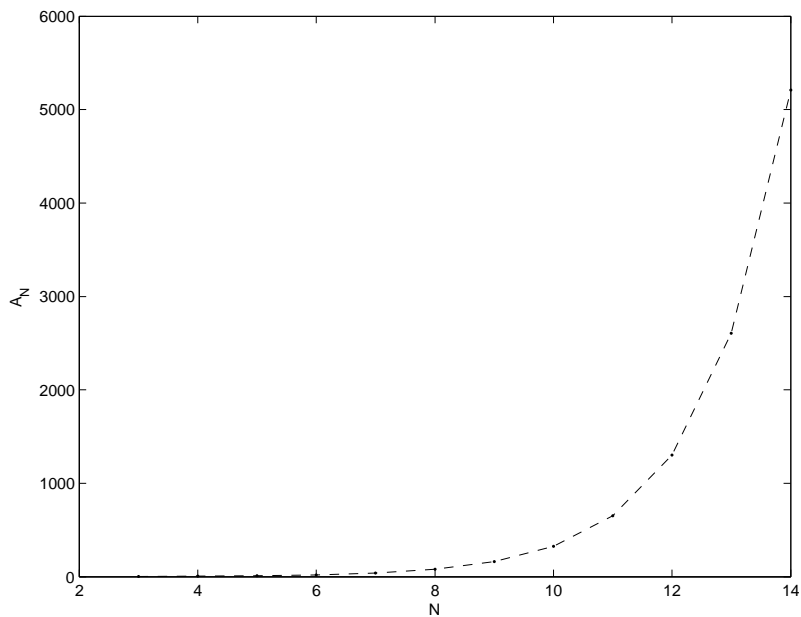}
	\caption{Maximal ratios of the distances between adjacent zeros}
	\label{fig:fig8}
\end{figure}
\begin{figure}[!htb]
	\centering
	\includegraphics[scale=.6]{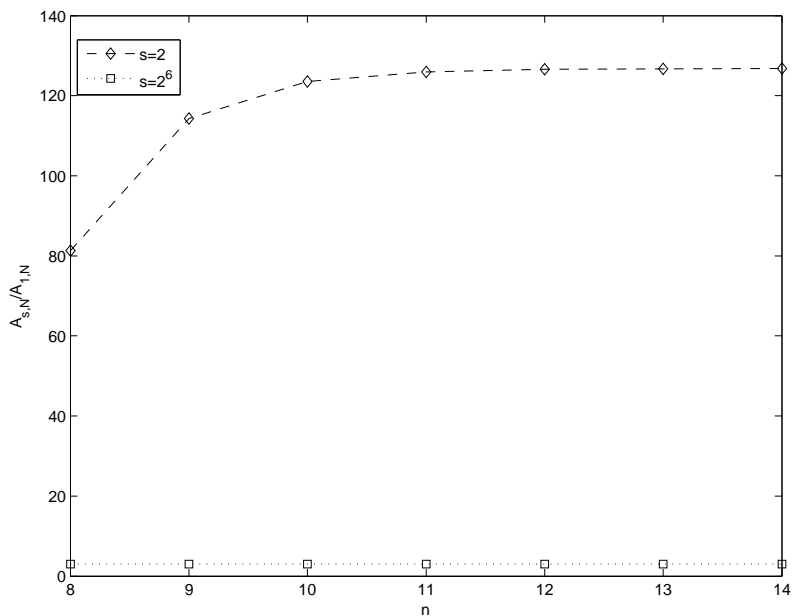}
	\caption{Ratios of the distances between prescribed adjacent zeros}
	\label{fig:fig9}
\end{figure}
For a measure $\mu$ having support on $\mathbb{R}$, let $Z_n(\mu):=\{x: Q_n(x;\mu)=0\}$. For $n> 1$ with $n\in\mathbb{N}$, we define $M_n(\mu)$ by
\begin{equation*}
	M_n(\mu):= \inf_{\substack{x,x^\prime\in Z_n(\mu)\\ x\neq x\prime}}|x-x^\prime|.
\end{equation*}
For a given $\gamma=(\gamma_k)_{k=1}^\infty$ let us enumerate the elements of $Z_N(\mu_{K(\gamma)})$ by $x_{1,N}<\dots <x_{N,N}$. The behavior of $	\left(M_N\left(\mu_{K(\gamma)}\right)\right)_{N=1}^\infty$, in other words, the global behavior the spacing of the zeros, were investigated in \cite{g1}. Here, we numerically study some aspects of the local behavior of the zeros.

We consider only Model 1 since the calculations give similar results for the other models. For $N=2^3, 2^4, \dots, 2^{14},$ let $A_{n,N}:= |x_{2n,N}-x_{2n-1,N}|$ where $n\in\{1,\dots, N/2\}$. We computed (see Figure ~\ref{fig:fig8}) $A_N:=\displaystyle\max_{n,m\in\{1,\ldots,N/2\}} \frac{A_{n,N}}{A_{m,N}}$ for each such $N$. 

$(A_{2^n})_{n=3}^{14}$ increases fast and this indicates that $(A_{2^n})_{n=2}^{\infty}$ is unbounded. 

For $N=2^{14}$ and  $s=2$, $s=6$ we plotted (see Figure ~\ref{fig:fig9}) $A_{s,N}/A_{1,N}$. These ratios tend to converge fast.

In the next conjecture, we exclude the case of small $\gamma$ for the following reason: Let $\gamma=(\gamma_k)_{k=1}^\infty$ satisfies $\sum_{k=1}^\infty \gamma_k=M<\infty$ with $\gamma_k\leq 1/32$ for all $k\in\mathbb{N}$ and $\delta_k:=\gamma_1\cdots\gamma_k$. Then $A_{j,2^k}\leq \exp{(16M)}\delta_{k-1}$ for all $k>1$ by Lemma 6 in \cite{gonc}. By Lemma 4 and Lemma 6 in \cite{gonc} we conversely have $A_{j,2^k}\geq (7/8)\delta_{k-1}$. Therefore $A_{2^k}\leq (8/7) \exp{(16M)}$. Hence, $(A_{2^n})_{n=2}^\infty$ is bounded. 
\begin{Conjecture}
	For each $\gamma=(\gamma_k)_{k=1}^\infty$ with $\inf_k \gamma_k>0$, $(A_{2^k})_{k=1}^\infty$ is an unbounded sequence. If $s=2^k$ for some $k\in\mathbb{N}$, there is a $c_0\in\mathbb{R}$ depending on $k$ such that $$\lim_{n\rightarrow\infty} \frac{A_{s,2^n}}{A_{1,2^n}}=c_0.$$
\end{Conjecture}

For the parameters $c>3$, $H_{\mu_{J(f)}}$ is almost periodic where $f(z)=z^2-c$, see \cite{bel2}. It was conjectured in \cite{belis} that $H_{\mu_{J(f)}}$ is always almost periodic as soon as $c>2$. For $c=2$, $H_{\mu_{J(f)}}$ is not almost periodic since $a_1=2$ and $a_n=1$ for $n\geq2$ but it is asymptotically almost periodic. Therefore if this conjecture is true then we have the following: $H_{\mu_{J(f)}}$ is almost periodic if and only if $J(f)$ is non Parreau-Widom.

We did not make any distinction between asymptotic almost periodicity and almost periodicity in Section 3 and Section 4 since these two cases are indistinguishable numerically. But we remark that if $\liminf a_n\neq 0$ then the asymptotics $\lim_{j\rightarrow\infty} a_{j\cdot 2^s+n}= a_n$ cease to hold immediately. We do not expect $H_{\mu_{K(\gamma)}}$ to be almost periodic for the Parreau-Widom case for that reason. For a parameter $\gamma=(\gamma_s)_{s=1}^\infty$ such that $\lim_{j\rightarrow\infty} a_{j\cdot 2^s+n}= a_n$ holds for each $s$ and $n$ it is likely that  $H_{\mu_{K(\gamma)}}$ is almost periodic. These asymptotics hold only for the non Parreau-Widom case but it is unclear that if these hold for all parameters making $K(\gamma)$ non Parreau-Widom.

Hausdorff dimension of a unit Borel measure $\mu$ supported on $\mathbb{C}$ is defined by $\mathrm{dim}(\mu):=\inf \{\mathrm{HD}(K): \mu(K)=1\}$ where $HD(\cdot)$ stands for the Hausdorff dimension of the given set. Hausdorff dimension of equilibrium measures were studied for many fractals (see \cite{makarov} for an account of the previous results) and in particular for autonomous polynomials Julia sets (see e.g. \cite{prz}). If $f$ is a nonlinear monic polyomial and $J(f)$ is a Cantor set then by p. 176 in\cite{prz} (see also p. 22 in \cite{makarov}) we have $\mathrm{dim}\left(\mu_{J(f)}\right)<1$. For $K(\gamma)$, $\sum_{s=1}^\infty \sqrt{\varepsilon_s}<\infty$ implies that $\mathrm{dim}\left(\mu_{K(\gamma)}\right)=1$ since $\mu_{(K(\gamma))}$ and the Lebesgue measure restricted to $K(\gamma)$ (see 4.6.1 in \cite{sodin}) are mutually absolutely continuous. Moreover, our numerical experiments suggest that $K(\gamma)$ has zero Lebesgue measure for non Parreau-Widom case. It may also be true that $\mathrm{dim}\left(\mu_{K(\gamma)}\right)<1$ for this particular case. Hence, it is an interesting problem to find a systematic way for calculating the dimension of equilibrium measures of $K(\gamma)$ and generalized Julia sets in general.


\begin{thebibliography}{155}


\bibitem{g1} Alpan, G: {Spacing properties of the zeros of orthogonal polynomials on Cantor sets via a sequence of polynomial mappings}, Preprint (2015), arXiv:1509.07391v2

\bibitem{alpgon3}Alpan, G., Goncharov, A.: {Two measures on Cantor sets}, J. Approx. Theory. \textbf{186}, 28--32 (2014)

\bibitem{alpgon4}Alpan, G., Goncharov, A.: {Widom factors for the Hilbert norm}, Banach Center Publ. \textbf{107}, 11--18 (2015) 

\bibitem{alpgon}Alpan, G., Goncharov, A.: {Orthogonal polynomials for the weakly equilibrium Cantor sets}, accepted for publication in Proc. Amer. Math. Soc. 

\bibitem{alpgon2}Alpan, G., Goncharov, A.: {Orthogonal polynomials on generalized Julia sets}, Preprint (2015), arXiv:1503.07098v3

\bibitem{alpgonhat}Alpan, G., Goncharov, A., Hatino\u{g}lu, B.: {Some asymptotics for extremal polynomials}, accepted for publication in "Computational Analysis: Contributions from AMAT 2015" in Springer-New York. 

\bibitem{andr} Andrievskii, V.V.: {Chebyshev Polynomials on a System of Continua}, Constr. Approx., (2015) doi:10.1007/s00365-015-9280-8

\bibitem{avil}Avila, A., Jitomirskaya, S.: {The Ten Martini problem}. Ann. of Math. \textbf{170}, 303--342 (2009)

\bibitem{Barnsley1}Barnsley M.F., Geronimo, J.S., Harrington, A.N.: {Orthogonal polynomials associated with invariant measures on Julia sets}. Bull. Amer. Math. Soc. \textbf{7}, 381--384 (1982)

\bibitem{Barnsley3}Barnsley M.F., Geronimo, J.S., Harrington, A.N.: {Infinite-Dimensional Jacobi Matrices Associated with Julia Sets}. Proc. Amer. Math. Soc. \textbf{88}(4), 625--630 (1983)

\bibitem{Barnsley4}Barnsley, M.F., Geronimo, J.S., Harrington, A.N.: {Almost periodic Jacobi matrices associated with Julia sets for polynomials}. Comm. Math. Phys. \textbf{99}(3), 303--317 (1985)

\bibitem{bel2}Bellissard, J., Bessis, D., Moussa, P.: {Chaotic states of almost periodic Schr\"{o}dinger operators}. Phys. Rev. Lett. \textbf{49}, 701--704 (1982)

\bibitem{belis} Bellissard, J., Geronimo, J., Volberg, A., Yuditskii, P: {Are they limit periodic?} Complex analysis and dynamical systems II, Contemp. Math., \textbf{382} , Amer. Math. Soc., Providence, RI, 43–53 (2005)


\bibitem{bes2} Bessis, D: {Orthogonal polynomials Pad\'{e} approximations, and Julia sets}, in: \textit{Orthogonal Polynomials: Theory \& Practice, 294} (P. Nevai ed.),  Kluwer, Dordrecht, 55--97 (1990)

\bibitem{besger} Bessis, D., Geronimo, J.S., Moussa, P.:{Function weighted measures and orthogonal polynomials on Julia sets}, Constr. Approx. \textbf{4}, 157--173 (1988)

	
\bibitem{Brolin}Brolin, H.: {Invariant sets under iteration of rational functions}, Ark. Mat. \textbf{6}(2), 103--144 (1965)
	
\bibitem{Bruck1}Br\"{u}ck, R.: {Geometric properties of Julia sets of the composition of polynomials of the form $z^2 +c_n$}, Pac. J. Math. \textbf{198}, 347--372 (2001)
	
\bibitem{Bruck} Br\"{u}ck, R., B\"{u}ger, M.: {Generalized iteration}, Comput. Methods Funct. Theory \textbf{3}, 201--252 (2003)
	
\bibitem{Buger} B\"{u}ger, M.: {Self-similarity of Julia sets of the composition of polynomials}, Ergodic Theory Dyn. Syst. \textbf{17}, 1289--1297 (1997)
	
\bibitem{christiansen}Christiansen, J.S.: {Szeg\H{o}'s theorem on Parreau-Widom sets}, Adv. Math. \textbf{229}, 1180--1204 (2012)

\bibitem{chriss}Christiansen, J.S., Simon, B., Zinchenko, M.: {Finite Gap Jacobi Matrices, I. The Isospectral Torus}. Constr. Approx. \textbf{32}, 1--65 (2009)

\bibitem{Chris}Christiansen, J.S., Simon, B., Zinchenko, M.: {Finite Gap Jacobi Matrices, II. The Szeg\"o Class}. Constr. Approx. \textbf{33}(3), 365--403 (2011)

\bibitem{csz}Christiansen, J.S., Simon, B., Zinchenko, M., {Asymptotics of Chebyshev Polynomials, I. Subsets of $\mathbb{R}$}, Preprint (2015), arXiv:1505.02604v1

\bibitem{delyon} Delyon, F., Souillard, B.:{The rotation number for finite difference operators and its properties}. Comm. Math. Phys.
\textbf{89}, 415--426 (1983)

\bibitem{dombrowski}Dombrowski, J.:{Quasitriangular matrices}. Proc. Amer. Math. Soc. \textbf{69}, 95--96 (1978)

\bibitem{fekete}Fekete, M.: Uber die Verteilung der Wurzeln bei gewissen algebraischen Gleichungen mit ganzzahligen 
Koeffizienten. Math. Z. \textbf{17}, 228–249 (1923) (in German)

\bibitem{gerhar}Geronimo, J.S., Harrell E.M. II, Van Assche, W.:. {On the asymptotic distribution of eigenvalues of banded matrices}. Constr. Approx. \textbf{4}, 403--417 (1988)

\bibitem{van assche}Geronimo, J.S., Van Assche, W.: {Orthogonal polynomials on several intervals via a polynomial mapping}, Trans. Amer. Math. Soc. \textbf{308}, 559--581 (1988)

\bibitem{golub} Golub, G.H., Welsch, J.H.:{Calculation of Gauss Quadrature Rules}, Math. Comp. \textbf{23}, 221--230 (1969)

\bibitem{gonc}Goncharov, A.: {Weakly equilibrium Cantor type sets}, Potential Anal. \textbf{40}, 143--161 (2014)

\bibitem{gonchat} Goncharov A., Hatino\u{g}lu, B.: {Widom Factors} , Potential Anal. \textbf{42}, 671--680 (2015)

\bibitem{Heilman}Heilman, S.M., Owrutsky, P., Strichartz, R.: {Orthogonal polynomials with respect to self-similar measures}. Exp. Math. \textbf{20}, 238-259 (2011)
	
	

\bibitem{kruger}Kr{\"u}ger, H., Simon, B.: {Cantor polynomials and some related classes of OPRL}. J. Approx. Theory \textbf{191}, 71--93 (2015) 
	

\bibitem{makarov} Makarov, N.:{Fine structure of harmonic measure}, St. Petersburg Math. J. \textbf{10}, 217–268 (1999)

\bibitem{Mane} Ma\~{n}\'{e}, R., Da Rocha, L.F.: {Julia sets are uniformly perfect}, Proc. Amer. Math. Soc. \textbf{116}(1), 251--257 (1992)

\bibitem{mant1}Mantica, G.: {A Stable Stieltjes Technique to Compute Jacobi Matrices Associated with Singular Measures}, Const. Approx. \textbf{12}, 509--530 (1996)

\bibitem{mant2}Mantica, G.: {Quantum Intermittency in Almost-Periodic Lattice Systems derived from their Spectral Properties}, Physica D \textbf{103}, 576--589 (1997)

\bibitem{mant3}Mantica, G.: {Numerical computation of the isospectral torus of finite gap sets and of IFS Cantor sets}, Preprint (2015), arXiv:1503.03801

\bibitem{mant4}Mantica, G.: {Orthogonal polynomials of equilibrium measures supported on Cantor sets}, J. Comput. Appl. Math. \textbf{290}, 239--258 (2015)


\bibitem{yuditskii}Peherstorfer, F., Volberg, A., Yuditskii, P.: {Limit periodic Jacobi matrices with a prescribed $p$-adic hull and a singular continuous spectrum}, Math. Res. Lett. \textbf{13}, 215--230 (2006)

\bibitem{peher} Peherstorfer, F., Yuditskii, P.: {Asymptotic behavior of polynomials orthonormal on a homogeneous
set}, J. Anal. Math. \textbf{89}, 113-154 (2003)

\bibitem{petersen} Petersen, K: {Ergodic Theory}, Cambridge Studies in Advanced Mathematics, Cambridge University Press, Cambridge (1983) 

\bibitem{pommerenke} Pommerenke, Ch.: {On the Green's function of Fuchsian groups}, Ann. Acad. Sci. Fenn. Ser. A I Math. \textbf{2}, 409–427 (1976)

\bibitem{prz} Przytycki, F.: {Hausdorff dimension of harmonic measure on the boundary of an attractive basin for a holomorphic map}, Invent. Math. \textbf{80}, 161–179 (1985)

\bibitem{Ransford}Ransford, T.: {Potential theory in the complex plane}, Cambridge University Press, (1995)

	
\bibitem{saff}Saff, E.B., Totik, V.: {Logarithmic potentials with external fields}, Springer-Verlag, New York (1997)

\bibitem{sif} Schiefermayr, K.: {A lower bound for the minimum deviation of the Chebyshev polynomial on a compact real set}, East J. Approx. \textbf{14}, 223--233 (2008)
	

\bibitem{Sim3}Simon, B.: {Szeg\H{o}'s Theorem and Its Descendants: Spectral Theory for $L^2$ Perturbations of Orthogonal Polynomials}, Princeton University Press, Princeton, NY (2011)
	
\bibitem{sodin}Sodin, M., Yuditskii, P.: {Almost periodic Jacobi matrices with homogeneous spectrum, infinite-dimensional Jacobi inversion, and Hardy spaces of character-automorphic functions}, J. Geom. Anal. \textbf{7}, 387--435 (1997)

\bibitem{szego}Szeg\H{o}, G.: Bemerkungen zu einer Arbeit von Herrn M. Fekete:  Uber die Verteilung der Wurzeln bei gewissen algebraischen Gleichungen mit ganzzahligen Koeffizienten. Math. Z. \textbf{21}, 203–208 (1924) (in
German)

\bibitem{teschl} Teschl, G.: {Jacobi Operators and Completely Integrable Nonlinear Lattices}, Math. Surv. and Mon. 72, Amer. Math. Soc., Rhode Island (2000)

\bibitem{tot11} Totik, V.: {Chebyshev constants and the inheritance problem}, J. Approx. Theory. \textbf{160}, 187--201 (2009)

\bibitem{tot2} Totik, V.: {Chebyshev Polynomials on Compact Sets}, Potential Anal. \textbf{40}, 511--524 (2014)

\bibitem{tot1} Totik, V., Yuditskii, P.: {On a conjecture of Widom}, J. Approx. Theory \textbf{190}, 50--61 (2015)

\bibitem{totikvar}Totik, V., Varga, T.: {Chebyshev and fast decreasing polynomials}. Proc. London Math. Soc. doi:10.1112/plms/pdv014

	
\bibitem{ase}Van Assche, W.: {Asymptotics for orthogonal polynomials}, Lecture Notes in Mathematics, 1265, Springer-Verlag, Berlin (1987)
	

\bibitem{widom}Widom, H.: {Polynomials associated with measures in the complex plane}. J. Math. Mech. \textbf{16}, 997--1013 (1967)

\bibitem{widom2}Widom, H: {Extremal polynomials associated with a system of curves in the complex plane}. Adv. Math. \textbf{3}, 127--232 (1969)

\bibitem{yuditskii}Yudistkii, P: {On the Direct Cauchy Theorem in Widom Domains: Positive and Negative Examples}, Comput. Methods Funct. Theory \textbf{11}, 395--414 (2012)
	
	
\end{thebibliography}
\end{document}